\documentclass[final,onecolumn]{elsarticle}
\usepackage{amsmath}
\usepackage{amsthm}
\usepackage{latexsym}
\usepackage{amssymb}
\usepackage{inputenc}
\usepackage{color}
\usepackage{psfrag}
\usepackage{tikz}
\usetikzlibrary{shapes,arrows, positioning}
\usetikzlibrary{arrows,calc}
\usepackage{comment}
\usepackage{pdfpages}
\usepackage{mleftright}
\usepackage{showkeys}
\usepackage{hyperref}
\usepackage{showlabels}
\usepackage{bbold}
\usepackage[toc,page]{appendix}
\usepackage{subcaption}
\usepackage{scalerel}
\usepackage{stackengine,wasysym}
\usepackage{float}
\usepackage{stackengine}
\usepackage{scalerel}
\usetikzlibrary{matrix}
\usepackage{tikz-cd}

\usepackage{blindtext}
\allowdisplaybreaks
\tikzset{
    block/.style = {draw, rectangle, 
        minimum height=1cm, 
        minimum width=2cm},	
    input/.style = {coordinate,node distance=1cm},
    output/.style = {coordinate,node distance=4cm},
    arrow/.style={draw, -latex,node distance=1.5cm},
    pinstyle/.style = {pin edge={latex-, black,node distance=2cm}},
    sum/.style = {draw, circle}
}

\newcommand{\R}{\mathbb{R}}
\newcommand{\C}{\mathbb{C}}
\DeclareMathOperator{\diag}{diag}

\newtheorem{example}{Example}[section]
\newtheorem{prop}[example]{Proposition}
\newtheorem{theorem}[example]{Theorem}
\newtheorem{defi}[example]{Definition}
\newtheorem{lemma}[example]{Lemma}

\newtheorem{rem}[example]{Remark}
\newtheorem*{theorem*}{Theorem}

\makeatletter
\def\ps@pprintTitle{%
 \let\@oddhead\@empty
 \let\@evenhead\@empty
 \def\@oddfoot{}%
 \let\@evenfoot\@oddfoot}
\makeatother

\begin{document}

\begin{frontmatter}



\title{Stability radius for infinite-dimensional interconnected systems}


\author[BUW]{Birgit Jacob\corref{cor}}
\ead{bjacob@uni-wuppertal.de}
 \cortext[cor]{corresponding author}
 \author[BUW]{Sebastian M\"oller}
 \ead{smoeller@uni-wuppertal.de}
 \author[BUW]{Christian Wyss}
 \ead{wyss@math.uni-wuppertal.de}
\address[BUW]{School of Mathematics and Natural Sciences, University of Wuppertal, Wuppertal, Germany}

\begin{abstract}
The stability radius for finitely many interconnected linear exponentially stable well-posed systems with respect to static perturbations is studied. If the output space of each system is finite-dimensional, then a lower bound for the stability radius in terms of the norm of the corresponding transfer functions is given. Moreover, for regular linear systems with zero feedthrough operator and finite-dimensional output spaces a formula for the stability radius is developed.
\end{abstract}


\end{frontmatter}

\section{Introduction}
This paper is concerned with  a finite number of exponentially stable well-posed linear systems $\Sigma_i$, $i=1,\ldots,N$, which are interconnected by a given structure. For the notion of well-posed linear system we refer the reader to Section \ref{sec:wpandregsystems}. Let $\Sigma:=\diag(\Sigma_1,\ldots,\Sigma_N)$, which  is again an exponentially stable  well-posed linear system. Clearly, $\Sigma$ represents the uncoupled system.
We assume that the magnitudes of the couplings bet\-ween the systems $\Sigma_i$ are uncertain. 
The matrix
$\mathcal{E}=(e_{ij})\in \R^{N \times N}$ describes the structure and the strength of the interconnection of the systems, that is, the entry $e_{ij}$ of  $\mathcal{E}$ can be interpreted as the strength of the  connection of the output of system $\Sigma_j$ to the input of system $\Sigma_i$. If $e_{ij}=0$, then the output of $\Sigma_j$ does not influence $\Sigma_i$.
We denote the input and output of the system  $\Sigma_i$ by $u_i$ and $y_i$, respectively, and  consider interconnections of the form  
$$
u_i(t)=\sum_{j=1}^N \Delta_{ij}e_{ij}y_j(t).
$$
where $\Delta_{ij}$ are unknown linear, bounded operators, describing the magnitude of the coupling.  For short we write $\Delta=(\Delta_{ij})$ and 
we denote the interconnected system by $\Sigma^{\Delta \circ \mathcal{E}}$.

 It is easy to see, that for operators $\Delta$ which are small in some norm, the interconnected system $\Sigma^{\Delta \circ \mathcal{E}}$ is again an exponentially stable well-posed linear system.
In applications it is natural to ask for the largest bound $r>0$ such that exponential stability is preserved for all magnitudes of the coupling $\Delta$ of norm strictly less than $r$ in a given normed perturbation set.
This largest bound $r>0$  is called the {\em stability radius}. 

The stability radius was introduced in 1986 by Hinrichsen and Pritchard \cite{HiPr86,hinrichsen1986stability} for finite-dimensional time-invariant systems. We note that for a fixed system the stability radius may depend on the normed set of perturbations and on the notion of stability. For finite-dimensional time-invariant systems there are formulas available for the stability radius with respect to different classes of perturbations, see  \cite{HiPr90} for a comprehensive survey. Pritchard and Townley introduced the stability radii for infinite-dimensional systems \cite{PrTo87,PrTo88}. 
The concept of the stability radius for interconnected systems was introduced by Hinrichsen and Pritchard in \cite{hinrichsen2009composite} for the finite-dimensional case. In this paper we consider interconnected infinite-dimensional systems and characterize the 
stability radius $r(\Sigma, \mathcal{E})$, which is given by
\begin{align*}
r(\Sigma,\mathcal{E})=
\sup &\{r>0 \mid \Delta \circ \mathcal{E} \text{ is an admissible feed-}\\
&\text{back and } \Sigma^{\Delta \circ \mathcal{E}} \text{ is exponentially stable}\\
&\text{for all } \Delta \text{ with } \|\Delta\|<r\}.
\end{align*}
Here the notion of an admissible feedback guarantees that $\Sigma^{\Delta \circ \mathcal{E}}$ is again a well-posed linear system, see Section \ref{sec:wpandregsystems}.  Moreover, the  perturbation class under consideration
(in particular the choice for the norm of $\Delta$)
will be introduced in Section \ref{sec:stabilityradiuswellposed}.


The main result of this paper is as follows. Let
$$
\resizebox{.47 \textwidth}{!}
{$
\Theta =\left[\sup\limits_{\omega \in \R}\rho\left(
 \begin{bmatrix}
\|\mathbf{G}_1(i\omega)\|^2&&&0\\
&\ddots&&\\0
&&&\|\mathbf{G}_N(i\omega)\|^2
\end{bmatrix}
\mathcal{E}^{\circ 2}
 \right)\right]^{\textstyle \frac{1}{2}}
$},
$$
where $\rho(\cdot)$ denotes the spectral radius and $\mathbf{G}_i$ is the transfer function of system $\Sigma_i$.

\begin{theorem}\label{thm:main}
If $\Sigma$ is an exponentially stable well-posed linear system and each system $\Sigma_i$ has a finite-dimensional output space, then 
\begin{align*}
\displaystyle r\left( \Sigma, \mathcal{E} \right)\ge\frac{1}{\Theta},
\end{align*}
If additionally $\Sigma$ is a regular linear system with feedthrough operator $0$, then the stability radius satisfies 
\begin{align*}
\displaystyle r\left( \Sigma, \mathcal{E} \right)=\frac{1}{\Theta}.
\end{align*}
\end{theorem}
This theorem comprises the results of Theorems \ref{thm:stabilityradiusforwps}  and \ref{thm:stabilityradiusforrs}. We note that  Theorem \ref{thm:main} is even new in the finite-dimensional setting as we
study the stability radius also for 
a norm that is different to the one 
in \cite{hinrichsen2009composite}.
Further, Theorem \ref{thm:main} generalizes the main result of \cite{PrTo87} slightly as in  \cite{PrTo87} by assumption system $\Sigma$ has feedthrough operator $0$.

We proceed as follows. In Section \ref{secpreliminaries} we summarize some known facts on positive matrices and Hardy spaces. The spectrum of multiplication operators are studied in Section
\ref{sec:mult}. A short introduction to well-posed linear and regular systems is given in 
\ref{sec:wpandregsystems}. The main results of this paper are formulated and proved in Sections
\ref{sec:stabilityradiuswellposed}
and \ref{sec:stabilityradiusregular}. 

\section{Preliminaries}\label{secpreliminaries}
In this section we give  a short summary of results  concerning the spectrum  and  the spectral radius for non-negative matrices as well as operator matrices which are useful for the proof of the main results of this paper. Moreover, we  summarize some results on Hardy spaces. 

A matrix $A=(a_{ij})_{i,j \in \{1,\ldots,N\}} \in \mathbb {R}^{N \times N}$,   $N \in \mathbb{N}$, is called  
{\em non-negative} (we briefly write $A \in \R_{\ge 0}^{N \times N})$, if  $a_{ij} \ge 0$ for every entry $a_{ij}$ of $A$. Further, $A$ is called 
{\em positive} ($A\in \R_{> 0}^{N \times N}$) if  $a_{ij}>0$ for every entry $a_{ij}$ of $A$.
Similarly, we define non-negative (resp. positive) vectors and denote  the set of such vectors  by  $\R_{\ge 0}^N$(resp. $ \R_{>0}^{N}$).
For a matrix  $A \in \R^{N \times N}$,  the {\em spectrum} $\sigma(A)$ is defined by
$\sigma(A)=\{ \lambda \in \C \mid \lambda$ is an eigenvalue of $A \}$ and $\varrho(A):=\C\setminus\sigma(A)$ denotes the {\em resolvent set} of $A$. Further,  we denote the {\em spectral radius} by 
$\rho(A)= \sup\{ |\lambda|  \mid \lambda \in \sigma(A)\}$.

\begin{lemma}[Perron-Frobenius {\cite[Chapter $8$]{horn1990matrix}}]
Suppose that  $A,B \in \R_{ \ge 0}^{N \times N}$. Then  the following results hold:
\begin{itemize}\label{lemmaperronforbenius}
\item[(i)] $\rho(A) \in \sigma(A)$ and   there exists a non-negative eigenvector $z$ of $A$ corresponding to the eigenvalue $\rho(A)$. The vector $z$ is called {\em Perron vector}.
\item[(ii)]If there exists  $\alpha \ge 0$  and $z \in \R_{ \ge 0}^N,~ z \neq 0, $  such that $Az \ge \alpha z$, then $\rho(A) \ge  \alpha$.
\item[(iii)] If there exists $\beta \ge 0$ and  $z \in \R_{>0}^N$ such that $Az \le \beta z$, then $\rho(A) \le \beta$.
\item[(iv)] If $A \le B$, i.e.~$a_{ij} \le b_{ij}$ for $i,j \in \{1, \ldots, N\}$ with $A =( a_{ij})$ and  $ B=( b_{ij} )$, then $\rho(A)\le \rho(B)$.
\end{itemize}
\end{lemma}
For  $A,B \in \R^{N \times N}$  the \emph{Hadamard product} $\circ$ of $A$ and $B$ is defined  by
\begin{align*}\label{eqn:hadamrdmatrix}
A \circ B &= \begin{bmatrix}
a_{11}&&\hdots&a_{1N}\\
\vdots&&\ddots&\vdots\\
a_{N1}&&\hdots&a_{NN}
 \end{bmatrix}
\circ 
\begin{bmatrix}
b_{11}&&\hdots&b_{1N}\\
\vdots&&\ddots&\vdots\\
b_{N1}&&\hdots&b_{NN}
 \end{bmatrix}\\
&:=
\begin{bmatrix}
a_{11} \cdot b_{11}&&\hdots&a_{1N}\cdot b_{1N}\\
\vdots&&\ddots& \vdots\\
a_{N1} \cdot b_{N1}&&\hdots&a_{NN} \cdot b_{NN}
\end{bmatrix}
\end{align*}
as the componentwise multiplication of the entries  of  both matrices.
For $A \circ A$ we briefly write $A^{\circ 2}$.

Let $X$ and $Y$ be  complex Banach spaces. We denote the set of all bounded linear operators from $X$ to $Y$  by $\mathcal{L}(X, Y)$. The operator norm of an operator $A \in \mathcal{L}(X,Y)$ is denoted by $\| \cdot \|$. In the case of $X= Y$ we briefly write $\mathcal{L}(X)$ for  $\mathcal{L}(X,X)$. 
For an operator $A \in \mathcal{L}(X)$ we denote its spectrum by $\sigma(A)$, its approximate point spectrum by $\sigma_{\mathrm{app}}(A)$, its resolvent by $\varrho(A)$ and its spectral radius by $\rho(A)$.\\
Suppose that $X_1,\ldots,X_N$ and $Y_1,\ldots,Y_N$ are complex Banach spaces and  $A=\left( A_{ij} \right)_{i,j \in \{1, \ldots, N \}}\in  \mathcal{L}(X, Y)$ with $X = \bigoplus_{i=1}^N X_i,~$ 
$Y= \bigoplus_{i=1}^N Y_i$ and $A_{ij} \in \mathcal{L}(X_j, Y_i)$. A matrix of this form is called  an \emph{operator matrix}. 

Let $\mathcal{E}=\left( e_{ij} \right)\in \R_{ \ge 0}^{N \times N}$  and $A$ an  operator matrix. The \rm{Hadamard product} of   $A$ with $\mathcal{E}$,  denoted by $\circ$, is defined via
\begin{equation*}
\circ: \mathcal{L}(X, Y) \times \R_{ \ge 0}^{N \times N} \rightarrow  \mathcal{L}(X, Y)
\end{equation*}
\begin{equation*}
A \circ \mathcal{E}=
\begin{bmatrix}
e_{11}A_{11}&&\dots&e_{1N}A_{1N}\\
\vdots&&\ddots&\vdots\\
e_{N1}A_{N1}&&\dots&e_{NN}A_{NN}
\end{bmatrix}.
\end{equation*}

Let $H$ be a complex separable  Hilbert space and $Z$ a Banach space.
For $\omega\in\R$ let $\C_\omega=\{s \in \C \mid  \mathrm{Re}(s)>\omega\}$.
We  define the \emph{Hardy spaces} $H^2(H)$ and $H^{\infty}(Z)$ by
\begin{align*}
\mathcal{H}^2(H):&=\{ f: \C_0 \rightarrow H \mid  f \mathrm{~is~holomorphic~and }\,\\
& \|f\|_2^2= \sup_{x >0} \frac{1}{2  \pi} \int_{-\infty}^{\infty}  \|f(x+i \omega)\|^2 d \omega < \infty\}
\end{align*} 
and
\begin{align*}
H^{\infty}(\C_{\omega};Z)&:= \{ G:\C_{\omega}\rightarrow Z \mid G \textrm{~is~holomorphic } \\
&\textrm{and~}\sup\limits_{Re(s) >\omega}\|G(s)\|< \infty \}
\end{align*}
\begin{lemma}[\cite{jacob2012linear}]\label{lemhut}
$\mathcal{H}^2(H)$ has the  following properties:
\begin{itemize}
\item For each $f \in \mathcal{H}^2(H)$ there exists a unique function $\tilde{f} \in L^2\left( (-i \infty, i \infty);H \right)$ such that 
$$
\lim_{x \searrow 0} f(x+ i \omega)= \tilde{f}(i \omega)
$$
for almost all $\omega \in \R$,
\item  $\lim_{x \searrow 0} \|f(x+ \cdot)- \tilde{f }(\cdot)) \|_{L^2\left( (-i \infty, i \infty);H \right)}=0$
\item The mapping $f \rightarrow \tilde{f}$ is linear, injective and satisfies
\[ 
\|f\|_2^2=\frac{1}{2 \pi} \int_{- \infty}^{\infty} \|\tilde{f}(i \omega)\|^2 d \omega.
\]
\item   $\mathcal{H}^2(H)$ is a Hilbert space with the inner product
\[
\langle f,g \rangle:=\frac{1}{2 \pi} \int_{- \infty}^{\infty} \langle \tilde{f}(i \omega) , \tilde{g}(i \omega) \rangle d \omega.
\]
\item Let $f \in \mathcal{H}^2(H)$ be a function different from the zero function. Then $\tilde f$ is non-zero almost everywhere on the imaginary axis.
\end{itemize}
\end{lemma}
For  a function $f \in L^2([0,\infty);H)$  the \emph{Laplace transform} is a function $\hat{f}$ defined by
\[
\hat{f}(s)= \int_0^{\infty} f(t) e^{-st}dt,~~s \in \C_0.
\]
\begin{theorem}[Paley-Wiener theorem \cite{thomas1997vector}]\label{paleywienertheorem}
The Laplace transform  is an isometric  isomorphism from $L^2([0,\infty);H)$  to $\mathcal{H}^2(H)$.
\end{theorem}

\section{The spectrum of a multiplication operator}\label{sec:mult}

Let $\left( X,\Sigma, \mu \right)$ be a $\sigma$-finite measure space  and  $q:X \rightarrow \C^{N \times N}$ a measurable matrix-valued  function.
The operator $A_q$ on $L^2(X;\C^N)$, defined  by  $A_q:f \mapsto qf$, i.e.
$$
\left( A_qf \right)(s)=q(s)f(s),\quad s\in X,
$$
for $f \in D(A_q)=\{f \in L^2(X;\C^{N }) \mid qf  \in L^2(X;\C^N)\}$
is called a  matrix multiplication operator.
We have the following relation between  the spectrum of the multiplication  operator $A_q$  and  the pointwise computed spectra of the function $q$. 
\begin{prop}[{\cite[Prop.~1]{holderrieth1991matrix}}]\label{thm:multop}
For $X= i \R,~\mu$ the Lebesgue-measure  and $q$ a continuous  and bounded matrix-valued function on $X$, we have 
$$
\sigma(A_q)= \overline{\bigcup_{\omega \in \R} \sigma(q(i \omega))}.
$$
\end{prop}

Let now $q \in \mathcal{H}^{\infty}(\C_{- \delta};\C^{N \times N})$, where $\delta>0$,
and let  $Y=\C^N$.
Let $A_q$ be the multiplication operator on $L^2(i\R;Y)$, defined as above.
Note that $A_q\in \mathcal{L}(L^2(i\R;Y))$ since $q$ is bounded.
Let $\tilde{A}_q\in\mathcal{L}(\mathcal{H}^2(Y))$ be the operator of multiplication by $q$
on $\mathcal{H}^2(Y)$, i.e.
$$
(\tilde{A}_qf)(s)=q(s)f(s),\quad s \in \C_0,
$$
for $f\in\mathcal{H}^2(Y)$. Then the following result holds.

\begin{lemma}\label{lem:multcom}
  $\sigma_{\mathrm{app}}(\tilde{A}_q) \subseteq \sigma(A_q)$.
\end{lemma}

\begin{proof}
  Consider the mapping $j:\mathcal{H}^2(Y) \rightarrow L^2(i\R;Y),~ jf=\tilde{f}$
  introduced in Lemma \ref{lemhut}.
  Note that $\|jf\|_{L^2(i\R;Y)}=\sqrt{2\pi}\|f\|_2$.
  The diagram
\begin{figure}[H]
	\begin{center}
	\includegraphics[ width=.4\textwidth]{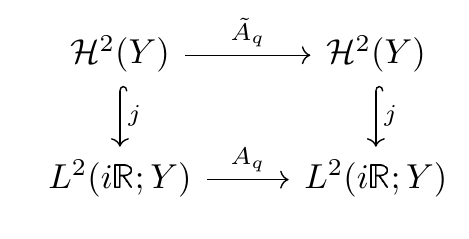}
	\label{fig:Aufg03_1}
	\end{center}
\end{figure}

is commutative since  for $f \in \mathcal{H}^2(Y)$
  \begin{align*}
    (A_q \tilde{f} )(i \omega)&=q(i\omega)\lim\limits_{r \searrow 0}f(r+i\omega)\\
    &= \lim\limits_{r \searrow 0}q(r+i\omega)f(r+i\omega)\\
    &=(j \tilde{A}_qf )(i \omega).
  \end{align*}
  Let $\lambda \in \sigma_{\mathrm{app}}(\tilde{A}_q)$.
  Then there exists a sequence $\left(f_n\right)_{n \in \mathbb{N}} \subseteq  \mathcal{H}^2(Y)$
  such that $\|f_n\|_2=1$ for $n \in \mathbb{N}$ and
  $\|(\lambda-\tilde{A}_q)f_n\|_2 \rightarrow 0$ for $n \to \infty$.
  Then $g_n=\frac{1}{\sqrt{2 \pi}}jf_n$
  satisfies  $\|g_n\|_{L^2(i\R;Y)}=1$
  and
  \begin{align*}
    \|(\lambda-A_q)g_n\|_{ L^2(i\R;Y)}
    &= \frac{1}{\sqrt{2 \pi}} \|(\lambda-A_q)j f_n\|_{ L^2(i\R;Y)}\\
    &= \frac{1}{\sqrt{2 \pi}}\|j \lambda f_n-j \tilde{A}_q f_n\|_{ L^2(i\R;Y)}\\
    & = \|\lambda f_n-\tilde{A}_qf_n \|_2 \overset{n \rightarrow \infty}{\longrightarrow} 0.
  \end{align*}
  Hence $\lambda\in \sigma_{\mathrm{app}}(A_q)\subseteq \sigma(A_q)$.
\end{proof}

\section{On well-posed  and regular systems}\label{sec:wpandregsystems}

In this section we provide a brief review of well-posed linear as well as regular linear systems. For more results we refer the reader to \cite{staffans2005well,tucsnak2014well}.
Let $Z$ be a  Banach space. For two functions $u,v \in L^2([0,\infty);Z)$ we define the $\tau$-concatenation $\diamond$ by
\[
u \underset{\tau}{\diamond}v=\mathbf{P}_{[0,\tau]}u+\mathbf{S}_{\tau}v,
\]  
where $\mathbf{P}_{[0,\tau]}$ denotes the truncation  of the function $u \in L^2([0,\infty);Z)$ to the interval $[0,\tau]$ while $\mathbf{S}_{\tau}$ denotes the operator of  right shift by $\tau$.
\begin{defi}[Well-posed linear system]\label{def:wps}
Let $X,U$ and $Y$ be Banach spaces. A {\em well-posed linear system} $\Sigma=\left( \mathbb{T},\Phi,\Psi,\mathbb{L} \right)$ on $(X,U,Y)$ is a  family of operators, where 
$\mathbb{T}=\left( T(t) \right)_{t \ge 0}$  is a $C_0$-semigroup on $X$,
$\Phi=(\Phi_t)_{t \ge0}$ with  $\Phi_t \in \mathcal{L}(L^2([0,\infty);U),X)$ such that 
$$
\Phi_{\tau+t}(u \underset{\tau}{\diamond} v)=T(t)\Phi_{\tau}u+\Phi_tv 
$$
for $u,v \in L^2([0,\infty);U), t, \tau\ge 0$,
$\Psi=(  \Psi_t )_{t \ge0}$ with $ \Psi_t\in \mathcal{L}(X,L^2([0,\infty);Y))$ such that
$$
\Psi_{\tau+t}x=\Psi_{\tau}x \underset{\tau}{\diamond} \Psi_tT(\tau)x
$$
for $x \in X, t,\tau \ge 0, \Psi_0=0$, and \\
$\mathbb{L}=( \mathbb{L}_t )_{t\ge0}$ with
$\mathbb{L}_t \in  \mathcal{L}(L^2([0,\infty);U), L^2([0,\infty);Y))$
satisfying 
$$
\mathbb{L}_{t+\tau}(u \underset{\tau}{\diamond} v)=\mathbb{L}_{\tau}u \underset{\tau}{\diamond}(\Psi_t\Phi_{\tau}u+\mathbb{L}_tv)
$$
for $u,v \in L^2([0,\infty);U), t, \tau \ge 0, \mathbb{L}_0=0$.
\end{defi}
We  call $X$ the {\em state space}, $U$ the {\em input space}  and $Y$ the {\em output  space} of $\Sigma$. The operators $\Phi_t$ are called {\em input operators}, the operators $\Psi_t$ are called {\em output operators} whereas the operators $\mathbb{L}_t$ are called {\em  input-output operators}.
Given an initial state $x_0 \in X$ and an input $u \in L^2_{\mathrm{loc}}([0,\infty);U)$ the state and output trajectories  $x:[0,\infty) \rightarrow X$ and $y:[0,\infty) \rightarrow Y$ of $\Sigma$ are defined by
\begin{equation}\label{wps_matrix}
\begin{bmatrix}
x(t)\\
\mathbf{P}_{[0,t]}y
\end{bmatrix}=
\Sigma_t
\begin{bmatrix}
x_0\\
\mathbf{P}_{[0,t]}u
\end{bmatrix}
\end{equation}
with 
\[
\Sigma_t=\begin{bmatrix}
T(t)&\Phi_t\\
\Psi_t&\mathbb{L}_t
\end{bmatrix}.
\]
Let $X_1$ be   the space $D(A)$ equipped with  norm $\|x\|_1=\|(\beta I-A)x\|$, where $A$ is the generator of the $C_0$-semigroup $(T(t))_{t\ge 0}$ and  $\beta \in \varrho(A)$ fixed.  The space $X_{-1}$ denotes  the  completion of X with respect to the norm $\|x\|_{-1}=\|(\beta I-A)^{-1} x\|$. Then $X_1\subseteq X \subseteq X_{-1}$ are continuous dense embedded. Note that different $\beta$ yield equivalent norms $\|\cdot\|_1$ and $\|\cdot\|_{-1}$. 
Each operator $T(t)$, $t\ge 0$, can be uniquely  extended to a linear bounded operator on $X_{-1}$. We denote this extension by $T_{-1}(t)$ and we note that  $\mathbb{T}_{-1}=\left( T_{-1}(t) \right)_{t \ge 0}$ is  a $C_0$-semigroup on $X_{-1}$.
The generator $A_{-1}$ of $\mathbb{T}_{-1}$  has domain   $D(A_{-1})=X$ and is  an extension of $A$ to $X_{-1}$.\\
For a well-posed linear  system $\Sigma$ there exists a unique  operator $B \in  \mathcal{L}(U,X_{-1})$, the {\em control operator} of $\Sigma$, such that for every  $t \ge 0$ the operator  $\Phi_t$ can be represented via
\begin{equation}\label{eqn:int_phi_t}
\Phi_{t} u=\int_0^{t} T_{-1}(t-s)Bu(s) ds.
\end{equation}
We remark that the integration in $\eqref{eqn:int_phi_t}$ is in $X_{-1}$, but the value of the  integral is an element of $X$.
Moreover, we obtain the existence of a unique operator $\Psi_{\infty}:X \rightarrow L^2_{\rm{loc}}([0,\infty);Y)$, the {\em extended output operator}, with $\mathbf{P}_{[0,t]} \Psi_{\infty}=\Psi_t$ for every $t \ge 0$.
It can be shown that there exists a unique  operator $C \in \mathcal{L}(X_1,Y)$ such that 
\begin{align*}
(\Psi_{\infty}x_0)(t)=CT(t)x_0, \qquad x_0\in X_1.
\end{align*}
Similarly, there exists a uniquely determined  operator  $\mathbb{L}_{\infty}: L^2_{\rm{loc}}([0,\infty);U) \rightarrow  L^2_{\rm loc}([0,\infty);Y)$, the {\em extended input-output operator}, such that $\mathbf{P}_{[0,t]}\mathbb{L}_{\infty}=\mathbb{L}_t\mathbf{P}_{[0,t]}$ for $t \ge 0$.
Using the Laplace transform we are able to represent $\mathbb{L}_{\infty}$ by the {\em transfer function} $\mathbf{G}$ of  $\Sigma$.  Let $\omega > \omega_0(\mathbb{T})$, where $\omega_0(\mathbb{T})$ denotes the growth bound of the $C_0$-semigroup $\mathbb{T}$ and $L^2_{\omega}([0,\infty);U)=e_\omega L^2([0,\infty);U)$ with $(e_\omega )v(t)=e^{\omega t}v(t)$ and $\|e_\omega v\|_{L^2_\omega}=\|v\|_{L^2}$.
$\mathbf{G}$ is a bounded analytic $\mathcal{L}(U,Y)$-valued   function on $\C_{\omega}$. For $s \in \C_{\omega_0(\mathbb{T})} $ and $u \in L^2_{\omega}([0,\infty);U)$ the Laplace-integral of $\mathbb{L}_{\infty}u$ at $s$ converges absolutely and 
\begin{equation*}
( \widehat{\mathbb{L}_{\infty}u} )(s)=\mathbf{G}(s)\hat{u}(s)
\end{equation*}
for $Re(s) > \omega$.
Furthermore, for $\alpha,\beta \in \C_{\omega_0({T})}$ the transfer function $\mathbf{G}$ satisfies  
\begin{align*}
&\mathbf{G}(\alpha)-\mathbf{G}(\beta)\\
&=(\beta-\alpha)C(\beta I-A_{-1})^{-1}(\alpha I-A_{-1})^{-1}B\\
&=C \left( (\alpha I-A_{-1})^{-1}-(\beta I -A_{-1})^{-1} \right) B.
\end{align*}
If the $C_0$-semigroup $(T(t))_{t\ge 0}$ is exponentially stable, i.e. $\omega_0(\mathbb{T}) <0$, then  $\Psi_\infty\in \mathcal{L}(X,L^2([0,\infty);Y))$ and $\mathbb{L}_{\infty}\in \mathcal{L}(L^2([0,\infty);U),L^2([0,\infty);Y))$ with
\begin{equation}\label{linfinitynorm}
\|\mathbb{L}_{\infty}\|=\sup_{\omega \in\R} \|\mathbf{G}(i \omega)\|
\end{equation}
using  the Paley-Wiener theorem  and the maximum modulus principle.

\begin{rem}\label{phiinL2}
Under the assumption that the $C_0$-semigroup of the well-posed linear system is exponentially stable, the  operator\\
$\Phi_\infty: L^2_{\rm{loc}}([0,\infty);U) \rightarrow  L^2_{\rm loc}([0,\infty);X)$, defined by 
\[ (\Phi_\infty u)(t) :=\Phi_t u,\]
satisfies 
\[
\Phi_\infty\in \mathcal{L}(L^2([0,\infty);U),L^2([0,\infty);X)).
\]
Indeed, considering $\Sigma(\mathbb{T},\Phi,\Psi,\mathbb{L})$ with an exponentially stable $C_0$-semigroup it is, as mentioned above, true that $\mathbb{L}_{\infty} \in \mathcal{L}(L^2([0,\infty),U),L^2([0,\infty),Y))$.
    
Defining the new well-posed system
$\widetilde{\Sigma}=\left(\mathbb{T},\Phi,\widetilde{\Psi},\widetilde{\mathbb{L}}  \right)$
with  $\mathbb{T}$ and $\Phi$ as in the original system,
$\Psi $ replaced  by
$$
(\widetilde{\Psi}_t x)(\tau):= \mathbb{1}_{[0,t]}(\tau)T(\tau)x,
$$
and  $\mathbb{L}$ replaced by
$$
( \widetilde{\mathbb{L}}_t u)(\tau):=\mathbb{1}_{[0,t]}(\tau)\Phi_{\tau}u,
$$
 we get
$\Phi_\infty=\widetilde{\mathbb{L}}_{\infty} \in \mathcal{L}(L^2([0,\infty),U),L^2([0,\infty),X))$. 
\end{rem}

A  well-posed linear system together with a feedback law $u=Ky+v$, where  $K\in\mathcal{L}(Y,U)$, defines a closed loop system $\Sigma^K$, see Figure \ref{feedback}.
\begin{center}
\begin{figure}[H]
    \begin{center}
        \begin{tikzpicture}[auto, node distance=1.8cm,>=latex']
         \node [input, name=input] {};
        \node [input, right of=input](help){};
        \node [block, right of =help] (controller) {$\Sigma$};
        \node [output, right of=controller] (output) {};
        \node [block, below of=controller] (feedback) {$K$};
        \draw [->] (input) |-  node{~~~~~~~~$v$}(help);
        \draw [->] (help) -- node{$u$}(controller);
        \draw [->] (controller) -- node [name=y] {$\hspace{1.5cm}y$}(output);
        \draw [->] (y) |- (feedback);
        \draw [->] (feedback) -| node[pos=0.99] {} 
        node [near end] {} (help);
    \end{tikzpicture}
\caption{Closed-loop system $\Sigma^K$}\label{feedback}
    \end{center}
\end{figure}
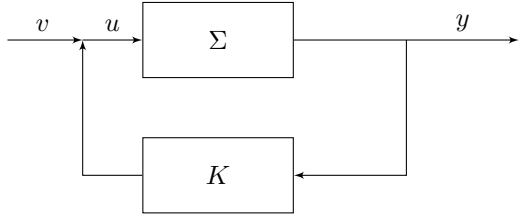
\end{center}
$K $ is called {\em admissible} if the closed-loop system   $\Sigma^K$ is also well-posed.
If $K$ is  admissible, the closed-loop system
$$
\Sigma_t^K=
\begin{bmatrix}
T^K(t)&\Phi_t^K\\
\Psi_t^K&\mathbb{L}_t^K
\end{bmatrix}
$$
 fulfills (see \cite{weiss1994regular})
\begin{equation}\label{zusammnhangfeedbackmirurspruenglichemsystem}
\Sigma^K_t-\Sigma_t=
\Sigma_t
\begin{bmatrix}
0&0\\
0&K
\end{bmatrix}
\Sigma^K_t, \qquad t \ge 0.
\end{equation}
For  explicit formulas  for  $\mathbb{T}^K,\Phi^K,\Psi^K$ and $\mathbb{L}^K$ we refer to  \cite[Thm.~7.1.2.]{staffans2005well}.

We note that 
if the semigroup of $\Sigma$ is exponentially stable and  the operator $I-\mathbb{L}_{\infty}K$ is
bijective on  $L^2([0,\infty);Y)$ then $K$ is admissible  (\cite[$\mathrm{Thm.}7.1.8$]{staffans2005well}).
Here, the operator $K$ in $I-\mathbb{L}_{\infty}K$ acts as
the multiplication operator induced by $K$.

A special case of well-posed linear systems are {\em regular linear systems}.
We call a well-posed linear system {\em regular}, if for every $v \in U$ the limit 
$$
Dv:= \lim_{\tau \rightarrow 0} \frac{1}{\tau} \int_0^\tau  \left(\mathbb{L}_{\infty}(\mathbb{1}_{[0,\infty)}v)\right)(s)ds
$$
  exists.  Then  $D \in \mathcal{L}(U,Y)$ and this operator   is called the {\em feedthrough operator} of the system  $\Sigma$.
  
For a well-posed linear system, we define the {\em Lebesgue extension} $C_L$ of $C \in \mathcal{L}(X_1,Y)$ by
$$
C_Lx=\lim_{\tau \searrow 0}C\frac{1}{\tau}\int_0^{\tau} T(s)xds
$$ 
with domain $D(C_L)$ defined as the 
set of all $x\in X$ for which the limit exists in $Y$. We note that $D(C)\subseteq D(C_L)$.
\begin{theorem}[{\cite[Thm.~5.6.5]{staffans2005well}}]
If $\Sigma$ is regular, then the output of $\Sigma$ is given by
\begin{equation}\label{eqn:outputregular}
y(t)=C_Lx(t)+Du(t).
\end{equation}
\end{theorem}
Thus a regular linear system with generating operators $A,B,C$ and  $D$ is completely determined by 
\begin{align*}
\dot{x}(t) &=A_{-1}x(t)+Bu(t)\\
y(t)&=C_Lx(t)+Du(t)
\end{align*}
Accordingly, we also denote the system by $\Sigma=\left(A,B,C,D\right)$.
For $\alpha \in \C_{\omega_0({T})}$ the transfer function $\mathbf{G}$ satisfies  
\begin{align*}
\mathbf{G}(\alpha)&=C_L(\alpha I-A_{-1})^{-1}B+D,
\end{align*}
in particular $(\alpha I-A)^{-1}Bu \in D(C_L)$ for $u \in U$.

As in the well-posed case we are interested in the closed loop system generated by an admissible feedback operator $K \in \mathcal{L}(Y,U)$. If $K$ is admissible, then the semigroup  corresponding to the closed-loop system is generated by $A^K: D(A^K)\subseteq X\rightarrow X$ with
$$
A^Kx=\left( A_{-1}+BK(I-DK)^{-1}C_L  \right)x,~ x\in D(A^K),
$$
\begin{align*}
D(A^K)=&\{x \in D(C_L) \mid  \\
& (A_{-1}+BK(I-DK)^{-1}C_L)x \in X\}.
\end{align*}

\section{Stability radius for well-posed  systems}\label{sec:stabilityradiuswellposed}

Let $X_1, \ldots, X_N$, $U_1, \ldots, U_N$,  $Y_1, \ldots,Y_N$ be Hilbert spaces. Furthermore, we assume that $\Sigma_i=( \mathbb{T}_i,\Phi_i,\Psi_i, \mathbb{L}_i )$, $i=1,\ldots,N$, are well-posed, exponentially stable  linear systems, i.e. for $i=1,\ldots,N$, $\Sigma_i$ is a well-posed linear system and $\mathbb{T}_i$ is exponentially stable.  We will now investigate the stability of interconnections of these systems. 
Let   
\begin{equation}\label{eqn:sum}
 X=\bigoplus_{i=1}^N X_i, \quad  Y=\bigoplus_{i=1}^N Y_i  \quad \mbox{ and }\quad  U=\bigoplus_{i=1}^N U_i
\end{equation}
 equipped with the norms 
\[ \|x\|^2=\sum_{i=1}^N \|x_i\|^2,~ \|y\|^2=\sum_{i=1}^N \|y_i\|^2, ~  \|u\|=\max_{i=1}^N\|u_i\|. \]
It is easy to see that $\Sigma:=(\mathbb{T},\Phi,\Psi,\mathbb L)$, defined by\\
 $(\diag(\mathbb{T}_i), \diag(\Phi_i),\diag(\Psi_i),\diag(\mathbb L_i))$ is an exponentially stable  well-posed linear system as well.\\
We consider interconnections of the form  
$$
u_i(t)=\sum_{j=1}^N \Delta_{ij}e_{ij}y_j(t).
$$
with $\Delta_{ij}\in\mathcal{L}\left( Y_j, U_i \right)$ and $e_{ij} \ge 0$.
Setting $\mathcal{E}=\big(e_{ij}\big)_{i,j=1,\ldots,N}$,
$\Delta=\left( \Delta_{ij} \right)_{ij}\in \mathcal{L}(Y,U)$,
$u=\left[\begin{smallmatrix}u_1\\\vdots\\u_N \end{smallmatrix}\right]$ and
$y=\left[\begin{smallmatrix} y_1\\\vdots\\y_N \end{smallmatrix}\right]$
this leads to  $u=\left(\Delta \circ \mathcal{E}  \right)y$ and hence to   the closed-loop system  $\Sigma^{\Delta \circ \mathcal{E}}$ as described in Figure \ref{feedback} with $K=\Delta \circ \mathcal{E}$.\\
The matrix $\mathcal{E}$ describes the structure and the strength of the interconnection of the systems. More precisely, the entry $e_{ij}$ of  $\mathcal{E}$ can be interpreted as the strength of the  connection of the output of system $\Sigma_j$ to the input of system $\Sigma_i$. If $e_{ij}=0$, then the output of $\Sigma_j$ does not influence $\Sigma_i$.
\begin{example}[construction of $\mathcal{E}$]
Consider four subsystems with an interconnection structure given by Figure \ref{exampl}.
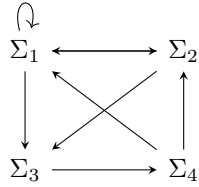
\begin{figure}[H]
\begin{center}
\hspace{-.4cm}
\begin{tikzpicture}
\matrix (m) [matrix of math nodes,row sep=3em,column sep=4em,minimum width=2em]
  {
     \Sigma_1 & \Sigma_2 \\
     \Sigma_3 & \Sigma_4 \\};
  \path[-stealth]
    (m-1-1) edge node [left] {} (m-2-1)
            edge  node [below] {} (m-1-2)
    (m-2-1.east|-m-2-2) edge node [below] {}
            node [above] {} (m-2-2)
    (m-1-2) edge [] (m-2-1)
            edge [] (m-1-1)
    (m-1-1) edge [loop above] (m-1-1)
    (m-2-2) edge node[above]{} (m-1-2)
    (m-2-2) edge node [above]{} (m-1-1);
\end{tikzpicture}
\caption{Example of an interconnection structure}\label{exampl}
\end{center}
\end{figure}
Then the matrix $\mathcal{E}$ is given by 
$$
 \mathcal{E}=\begin{bmatrix}
e_{11}&e_{12}&0&e_{14}\\
e_{21}&0&0&e_{24}\\
e_{31}&e_{32}&0&0\\
0&0&e_{43}&0
\end{bmatrix}.
$$
\end{example}

When determining the stability of the interconnected systems we regard
$\mathcal{E}$ and $\Sigma_i,~i=1,\ldots,N$ as fixed, while $\Delta $
is unknown.  Clearly, for $\Delta=0$ the system
$\Sigma^{\Delta \circ \mathcal{E}}$ is exponentially stable and one may expect that this
remains true if $\Delta$ is small.

We are interested in the largest $r$ such that
$\Sigma^{\Delta \circ \mathcal{E}}$ is exponentially stable for all $\Delta$ 
of norm smaller than $r$.  This leads to the concept of the stability radius;
for finite-dimensional interconnected systems it
was introduced by Hinrichsen and
Pritchard in \cite{hinrichsen2009composite}.

We consider the stability radius for
two different norms of $\Delta$:
The first is the operator norm $\|\Delta\|$, which is induced by
the chosen norms on $U$ and $Y$
and satisfies
\[\|\Delta\|=\max_{i=1}^N\|\Delta_i\|,\]
where $\|\Delta_i\|$ is the operator norm of the $i$th row
$\Delta_i=[\Delta_{i1},\dots,\Delta_{iN}]\in\mathcal{L}(Y,U_i)$ of $\Delta$.
The second norm is given by
\begin{equation}\label{norm:delta}
  \|\Delta\|_{2,\infty}=
  \max_{i=1}^N \left( \sum_{j=1}^N \|\Delta_{ij}\|^2 \right)^{\frac{1}{2}}.
\end{equation}
Note that both norms are related by
$\|\Delta\|\leq\|\Delta\|_{2,\infty}$.

\begin{defi}\label{def:stabilityradiuswps}
Let 
$$
\Sigma= (\diag({\mathbb{T}}_i), \diag(\Phi_i),\diag(\Psi_i),\diag(\mathbb L_i))
$$
be a well-posed, exponentially stable linear system and $\mathcal{E}=\left( e_{ij} \right)\in \R_{\ge 0}^{N \times N}$.
Then the {\em stability radius} $r(\Sigma,\mathcal{E})$ with respect to the operator norm
$\|\Delta\|$ is defined by
\begin{align*}
  r(\Sigma,\mathcal{E})=
  \sup &\{r>0 \mid \Delta \circ \mathcal{E} \text{ is admissible and}\\
  &\Sigma^{\Delta \circ \mathcal{E}}
  \text{ is exponentially stable }\\
  &\text{for all } \Delta \in \mathcal{L}(Y,U) \text{ with } \|\Delta\|<r\}.
\end{align*}
The stability radius
$r_{2,\infty}(\Sigma,\mathcal{E})$
is defined in the same way
with respect to the norm $\|\Delta\|_{2,\infty}$.
\end{defi}
It is easy to see that
\[r(\Sigma,\mathcal{E})\leq r_{2,\infty}(\Sigma,\mathcal{E}).\]

\begin{theorem}\label{thm:stabilityradiusforwps}
If $\Sigma= (\diag({\mathbb{T}}_i), \diag(\Phi_i), \diag(\Psi_i), \diag(\mathbb L_i))$ is a well-posed exponentially stable  linear   system  and $Y$ is finite-dimensional,  then the stability radii satisfy
\begin{align*}
r_{2,\infty}(\Sigma,\mathcal{E})\geq r(\Sigma,\mathcal{E})\ge\frac{1}{\Theta},
\end{align*}
where
$$
\Theta =\left[\sup\limits_{\omega \in \R}\rho\left(
\left[ \begin{smallmatrix}
\|\mathbf{G}_1(i\omega)\|^2&&&0\\
&\ddots&&\\0
&&&\|\mathbf{G}_N(i\omega)\|^2
\end{smallmatrix}\right]
\mathcal{E}^{\circ 2}
 \right)\right]^{\textstyle \frac{1}{2}}
$$
and $\mathbf{G}_k$ denotes the transfer function of $\Sigma_k$. 
\end{theorem}

\begin{rem}
  We use the convention $1/0=\infty$ here.
  In particular, if $\Theta=0$ then we obtain $r(\Sigma,\mathcal{E})=\infty$, i.e.,
  the system remains stable for every $\Delta\in\mathcal{L}(Y,U)$.
  We note that $\Theta=0$ holds if the graph corresponding to $\mathcal{E}$ does not
  contain any cycles \cite[Remark~4.6]{hinrichsen2009composite}.

\end{rem}

\begin{proof}[Proof of Theorem~\ref{thm:stabilityradiusforwps}]
  As a first step we show that for $\|\Delta \|<\frac{1}{\Theta}$ the operator
  $I-\mathbb{L}_\infty(\Delta \circ\mathcal{E})$  is invertible in
  $\mathcal{L}\left(L^2([0,\infty);Y)\right)$.
By the Paley-Wiener theorem \ref{paleywienertheorem}  this is equivalent to the fact that
$I-\mathbf{G}(\cdot)(\Delta \circ\mathcal{E})$, with $\mathbf{G}=\diag(\mathbf{G}_1,\ldots, \mathbf{G}_N)$ is  invertible in $\mathcal{L}\left( \mathcal{H}^2(Y)\right)$. Consider $\tilde{A}_q=\mathbf{G}(\cdot)(\Delta \circ \mathcal{E})$ as a multiplication  operator on $\mathcal{H}^2(Y)$ and $A_q=\mathbf{G}(\cdot )(\Delta \circ \mathcal{E})$ as a multiplication operator on $L^2(i\R;Y)$.
By Lemma \ref{lem:multcom}  $\sigma_{\mathrm{app}}( \tilde{A}_q ) \subseteq \sigma(A_q)$.
Using the assumption that $Y$ is finite-dimensional, we can then apply
Proposition \ref{thm:multop} to obtain
\[
\sigma_{\mathrm{app}}(\tilde{A}_q)\subseteq
\overline{\bigcup\limits_{\omega \in \R} \sigma(\mathbf{G}(i\omega)
  (\Delta \circ \mathcal{E}))}.
\]
To prove the invertibility of $I-\mathbf{G}(\cdot)(\Delta \circ \mathcal{E})$, it is now sufficient to show
\begin{equation}\label{specinball}
  \overline{ \bigcup_{\omega \in \R}\sigma\left(\mathbf{G}(i\omega)(\Delta \circ \mathcal{E})  \right)}
\subseteq \overline{B_r(0)}
\end{equation}
for some $r<1$, where  $B_r(0)$ denotes the ball around zero with radius $r$.
Indeed, since
$\tilde{A}_q$ is bounded and the boundary of the spectrum is contained in the approximate point spectrum,
this implies $\sigma(\tilde{A}_q)\subseteq\overline{B_r(0)}$
and hence $1\in\varrho(\tilde A_q)$.

So let $\lambda \in \bigcup\limits_{\omega \in \R}\sigma \left( \mathbf{G}(i \omega)(\Delta \circ \mathcal{E})  \right)$.
Then there exist $\omega \in \R$ and $y \neq 0$ such that 
$$
\lambda y = \mathbf{G}(i \omega)(\Delta \circ \mathcal{E})y.
$$
Componentwise this implies for $k=1,\ldots,N$
\begin{align*}
  |\lambda|^2 \|y_k\|^2
  &= \|\mathbf{G}_k(i \omega) \sum\limits_{j=1}^N e_{kj}\Delta_{kj}y_j\|^2\\
  &\le \|\mathbf{G}_k(i \omega)\|^2
  \Bigl\|\Delta_k\begin{bmatrix} e_{k1}y_1\\\vdots\\e_{kN}y_N\end{bmatrix}\Bigr\|^2\\
  &\le \|\mathbf{G}_k(i \omega)\|^2\|\Delta_k\|^2\sum_{j=1}^N \|e_{kj}y_j\|^2\\
  &\le \|\Delta\|^2 \|\mathbf{G}_k(i\omega)\|^2 \sum_{j=1}^N e_{kj}^2\|y_j\|^2.
\end{align*}
With $z=\left( \|y_k\|^2 \right)_{k=1,\ldots,N}$ it follows that
$$
|\lambda|^2z \le \|\Delta\|^2\diag(\|\mathbf{G}_k(i\omega)\|^2)\mathcal{E}^{\circ 2}z.
$$
Lemma \ref{lemmaperronforbenius} now implies
\[
\frac{|\lambda|^2}{\|\Delta\|^2}
\le \rho \left(\diag(\|\mathbf{G}_k(i \omega)\|^2)\mathcal{E}^{\circ2}  \right)
\leq\Theta^2.
\]
Thus,
\(
|\lambda|\le\|\Delta\|\Theta 
\)
and  hence \eqref{specinball} holds
with $r=\|\Delta\|\Theta <1$.

Therefore $I-\mathbb{L}_{\infty}(\Delta \circ \mathcal{E})$ is invertible.
In particular, it follows that  $\Delta \circ \mathcal{E}$ is an admissible feedback operator.
Moreover, the  exponential stability of $(T(t))_{t\ge 0}$ implies
$\Psi_\infty\in \mathcal{L}(X,L^2([0,\infty);Y))$.
Thus   for every $x_0 \in X$ the equation
\begin{equation}\label{eqn:feedb}
y= \Psi_\infty x_0+ \mathbb{L}_{\infty}(\Delta \circ \mathcal{E})y
\end{equation}
has a unique solution $y =y_{x_0}\in L^2([0,\infty);Y)$ that fulfills 
\begin{equation}\label{eqn:absch}
\|y_{x_0}\| \le M \|x_0\|
\end{equation}
 for some $M > 0$,    $M$  independent of $x_0$. Thanks to \eqref{zusammnhangfeedbackmirurspruenglichemsystem}, we have 
$$y_{x_0}=\Psi^{\Delta \circ \mathcal{E}}_{\infty} x_0$$ for every $x_0\in X$.
 Since   $\Sigma^{\Delta \circ \mathcal{E}}$ is  well-posed,  the $C_0$-semigroup $(T^{\Delta \circ \mathcal{E}}(t))_{t\ge 0}$ of $\Sigma^{\Delta \circ \mathcal{E}}$ satisfies, see \eqref{zusammnhangfeedbackmirurspruenglichemsystem},
\begin{align*}
T^{\Delta \circ \mathcal{E}}(t)x_0&= T(t)x_0+\Phi_t((\Delta \circ \mathcal{E})\Psi^{\Delta \circ \mathcal{E}}_\infty x_0)\\
&= T(t)x_0+\Phi_t((\Delta \circ \mathcal{E})y_{x_0})
\end{align*}
$  t\ge 0,~x_0\in X$.

Since $\mathbb{T}$ is exponentially stable, Datko's theorem (cf.~\cite[\rm{Thm.}V.5.8.]{engel1999one})
yields a constant $M_0>0$ such that 
\begin{equation}\label{stabil}
 \int_{0}^{\infty}\| T(t)x_0\|^2dt\le M_0\|x_0\|^2
\end{equation}
for every $x_0\in X$.
Now inequality \eqref{eqn:absch}, Remark \ref{phiinL2} and inequality \eqref{stabil} imply  
\begin{align*}
  &\int_0^{\infty} \|T^{\Delta \circ \mathcal{E}}(t)x_0\|^2dt \\
  &\le2 \int_{0}^{\infty}\| T(t)x_0\|^2dt
  + 2 \int_0^{\infty} \|\Phi_t((\Delta \circ \mathcal{E})y_{ x_0}) \|^2dt\\
  &\le 2M_0\|x_0\|^2
  + 2 \int_0^{\infty} \|\Phi_\infty((\Delta \circ \mathcal{E})y_{ x_0})(t) \|^2dt\\
  & \le 2M_0\|x_0\|^2
  + 2 M_1\int_0^{\infty} \|y_{ x_0}(t) \|^2dt\\
  & \le 2(M_0+M_1M^2) \|x_0\|^2.
\end{align*}
By Datko's theorem, 
it follows that  $(T^{\Delta \circ \mathcal{E}}(t))_{t\ge 0}$ is exponentially stable.
Thus  $r(\Sigma,\mathcal{E}) \ge \frac{1}{\Theta}$.
\end{proof}

\section{Stability radius for regular  systems   }\label{sec:stabilityradiusregular}

In this section we give an upper bound for   the stability radius  of $N$ well-posed linear systems under the additional assumption that every subsystem $(\mathbb{T}_i,\Phi_i,\Psi_i,\mathbb L_i)$ is a regular linear  system with feedthrough operator $D_i=0$. More precisely, we  consider $N$ regular linear  systems $\Sigma_j=(A_j,B_j,C_j,0)$, $j=1\ldots,N$, of the form
\begin{align}\label{eqn:subsystems}
\dot{x}_j(t)&= A_jx_j(t)+B_ju_j(t),\\ \notag
y_j(t)&= (C_j)_Lx_j(t)
\end{align}
and assume that $A_j$ generates an exponentially stable $C_0$-semigroup $\left( T_j(t) \right)_{t \ge 0}$ on $X_j$,  $B_j \in \mathcal{L}(U_j,(X_j)_{-1})$ and $C_j\in \mathcal{L}(D(A_j),Y_j)$.  Clearly, 
\begin{align*}
A=\left[\begin{smallmatrix}
A_1 & &0\\
& \ddots &  \\
0 & & A_N
\end{smallmatrix}\right]:D(A):= \bigoplus\limits_{i=1}^N D(A_i) \subseteq X \rightarrow X
\end{align*}
generates an exponentially stable $C_0$-semigroup $\left(T(t) \right)_{t \ge 0}$ on $X$, where $X$ is defined by $\eqref{eqn:sum}$. Additionally we define
\begin{align*}
B&=
\begin{bmatrix}
B_1 & & 0\\
& \ddots & \\
0& &  B_N \end{bmatrix}
\in L(U,X_{-1}), \\
C&=
\begin{bmatrix}
C_1 & &0 \\
& \ddots & \\
0 & &  C_N 
\end{bmatrix}
\in \mathcal{L}(D(A),Y),
\end{align*}
where $U$ and $Y$ are defined by $\eqref{eqn:sum}$.  Note, that $X_{-1}=\bigoplus_{i=1}^N (X_i)_{-1}$. Then $\Sigma=(A,B,C,0)$ is a regular linear system. With $\Delta $, $\mathcal{E}$ and the couplings of the systems  defined as in Section \ref{sec:stabilityradiuswellposed}  the entire interconnected system can be represented via
\begin{equation}\label{eqn:entiresystemwithperturbation}
\dot{x}(t)=A^{\Delta \circ \mathcal{E}}x(t)=\left( A_{-1}+ B \left( \Delta \circ \mathcal{E} \right)C_L \right)x(t)
\end{equation}
with $C_L=\diag\left((C_1)_L,\ldots,(C_N)_L \right)$.

\begin{theorem}\label{thm:stabilityradiusforrs}
  If $\Sigma=(A,B,C,0)$ is a regular exponentially stable linear system,
  then the stability radii satisfy
\begin{align*}
 r(\Sigma,\mathcal{E})\leq r_{2,\infty}(\Sigma,\mathcal{E})\le\frac{1}{\Theta},
\end{align*}
with $\Theta$ defined as in Theorem \ref{thm:stabilityradiusforwps}.
If additionally, $Y$ is finite-dimensional, then 
\begin{align*}
  r( \Sigma, \mathcal{E})=r_{2,\infty}(\Sigma,\mathcal{E})=\frac{1}{\Theta}.
\end{align*}
\end{theorem}

\begin{proof}
  By the definition of the stability radii it suffices to show that
  for every $\delta \in (0,\Theta)$ there exists
  $\Delta \in \mathcal{L}(Y,U)$ with
  $\|\Delta\|_{2,\infty} \le \frac{1}{\Theta-\delta}$ such that
  either $\Delta \circ\mathcal{E}$ is not an admissible feedback
  or $A^{\Delta \circ \mathcal{E}}$ does not generate an
  exponentially stable $C_0$-semigroup.
  
  Let $\delta \in (0,\Theta)$ be arbitrary.
  Choose  $\alpha \in (0,1)$ with $\frac{\Theta- \delta}{\alpha}< \Theta$.
  By the definition of $\Theta$  there exist $\omega_0 \in \R$ such that
  \begin{align*}
    \lambda&:=\rho\left(\left[
      \begin{smallmatrix}
        \|\mathbf{G}_1(i\omega_0)\|^2&&&0\\
        &\ddots&&\\0
        &&&\|\mathbf{G}_N(i\omega_0)\|^2
      \end{smallmatrix}\right]
    \mathcal{E}^{\circ 2}
    \right)\\
    &\ge \left(\frac{\Theta-\delta}{\alpha} \right)^2.
  \end{align*}
  In particular $\lambda>0$.
  By Lemma \ref{lemmaperronforbenius} there exists a vector $z \in \R^N_{\ge 0}, z \neq 0$,
  such that 
  \[
  \begin{bmatrix}
    \|\mathbf{G}_1(i\omega_0)\|^2&&&0\\
    &\ddots&&\\0
    &&&\|\mathbf{G}_N(i\omega_0)\|^2
  \end{bmatrix}
  \mathcal{E}^{\circ 2}z=\lambda z.
  \]
  Componentwise this reads 
  \begin{equation}\label{GksumzgleichThetaz}
    \|\mathbf{G}_k(i \omega_0)\|^2\sum_{l=1}^N e_{kl}^2z_l=\lambda z_k.
  \end{equation}

  Now for every $k\in\{1,\dots,N\}$ we choose
  $u_k\in U_k$ such that
  $\|\mathbf{G}_k(i \omega_0)u_k\|\ge \alpha \|\mathbf{G}_k(i \omega_0)\|\|u_k\|$
  and set $y_k=\mathbf{G}_k(i \omega_0)u_k$.
  If $z_k \neq 0$ then also $\|\mathbf{G}_k(i \omega_0)\| \neq 0$
  by $\eqref{GksumzgleichThetaz}$, and we can choose the norm of $u_k$ such that
  $z_k=\|y_k\|^2$.
  If $z_k=0$ then we set  $u_k=y_k=0$.
  
  For $k,l\in\{1,\ldots,N\}$ we define
  \begin{equation}\label{constructionofDelta}
    \Delta_{kl}=\begin{cases}
    \, \frac{u_k e_{kl} \langle \cdot,y_l\rangle}{\sum_{j=1}^N e^2_{kj} \|y_j\|^2}
    &\text{if } \sum_{j=1}^N e^2_{kj} \|y_j\|^2 \neq 0, \\
    \,  0 &\text{otherwise.} 
    \end{cases}
  \end{equation}
  Next we set
  $x_k=\left(i \omega_0 I-(A_k)_{-1} \right)^{-1}B_ku_k$
  and
  \(x=\left[\begin{smallmatrix}
      x_1\\ \vdots\\x_N
    \end{smallmatrix}\right]
  \). 
  Note that $x_k\in D((C_k)_L)$ and
  $(C_k)_Lx_k=\mathbf{G}_k(i\omega_0)u_k=y_k$.
  Hence $x_k\neq0$ if $z_k\neq0$ and thus $x\neq0$.

  Now we show that
  \begin{equation}\label{eqn:eigenv-comp}
    (A_k)_{-1}x_k+B_k\sum\limits_{l=1}^N e_{kl} \Delta_{kl} \left(C_l\right)_Lx_l
    =i \omega_0 x_k
  \end{equation}
  for all $k \in \{1,\ldots,N\}$.
  If $\sum_{j=1}^N e_{kj}^2 \|y_j \|^2 \neq 0$, then indeed
  \begin{align*}
    (&A_k)_{-1}x_k+B_k\sum_{l=1}^N e_{kl}\Delta_{kl}\left(C_l\right)_Lx_l \\
    &=(A_k)_{-1}x_k+B_k\sum_{l=1}^N e_{kl}
    \frac{u_ke_{kl} \|y_l\|^2}{\sum_{j=1}^N e_{kj}^2 \|y_j\|^2}\\ 
    &=(A_k)_{-1}x_k+B_ku_k=i \omega_0 x_k.
  \end{align*}
  If  $\sum_{j=1}^N e_{kj}^2\|y_j\|^2=0$, then
  $\Delta_{kl}=0$ for all $l \in \{1,\ldots,N\}$. Moreover
  $z_k=0$ by $\eqref{GksumzgleichThetaz}$, which in turn implies
  $x_k=0$ by construction. Hence \eqref{eqn:eigenv-comp} holds too.
  We have thus shown that
  \begin{equation}\label{eqn:eingenv}
    (A_{-1}+B(\Delta \circ \mathcal{E}) C_L )x= i \omega_0 x.
  \end{equation}

  Since $i\omega_0 x\in X$ it follows that
  $x \in D(A^{\Delta \circ \mathcal{E}})$.
  If $\Delta \circ \mathcal{E}$ is admissible,
  equation $\eqref{eqn:eingenv}$ implies that
  $A^{\Delta \circ \mathcal{E}}$ does not generate an exponentially stable
  $C_0$-semi\-group.

  It remains to show that $\|\Delta\|_{2,\infty} \le \frac{1}{\Theta-\delta}$.

If $z_k \neq 0$, then $\eqref{GksumzgleichThetaz}$ yields
$\sum_{l=1}^N e_{kl}^2 \|y_l\|^2 \neq 0$ and $\|\mathbf{G}_k(i \omega_0)\| \neq 0$,
and we compute 
\begin{align*}
  \sum_{l=1}^N \|\Delta_{kl}\|^2
  &= \sum_{l=1}^N \left( \frac{\|u_k\| \cdot e_{kl} \|y_l\|}{\sum_{j=1}^N e_{kj}^2 \|y_j\|^2}\right)^2\\
&= \sum_{l=1}^N \frac{\|u_k\|^2\cdot e_{kl}^2 \|y_l\|^2}{\left(\sum_{j=1}^N  e_{kj}^2\|y_j\|^2\right)^2}\\
&= \frac{\|u_k\|^2}{\sum_{j=1}^N e_{kj}^2 \|y_j\|^2}\\
&\le \frac{\|\mathbf{G}_k(i\omega_0)u_k\|^2}{\alpha^2\|\mathbf{G}_k(i \omega_0)\|^2 \sum_{j=1}^N e_{kj}^2z_j}\\
&= \frac{z_k}{\alpha^2\|\mathbf{G}_k(i \omega_0)\|^2 \sum_{j=1}^N e_{kj}^2z_j}\\
&=\frac{1}{\alpha^2  \lambda}
\le \frac{1}{ (\Theta- \delta)^2}.
\end{align*}
If $z_k=0$, then $u_k=0$ and therefore $\Delta_{kl}=0$ for $l \in \{1,\ldots,N\}$,\\
i.e.\ $\sum_{l=1}^N \|\Delta_{kl}\|^2=0$.
Consequently,
\[
\|\Delta\|_{2,\infty}= \max_{k=1,\ldots,N}\left( \sum_{l=1}^N \|\Delta_{kl}\|^2 \right)^{\frac{1}{2}} \le \frac{1}{\Theta-\delta}.
\]
The statement for finite-dimensional output spaces $Y$
now follows from Theorem \ref{thm:stabilityradiusforwps}.
\end{proof}


\section{Conclusion}

In this paper we studied the  stability radius for finitely many interconnected linear exponentially stable well-posed systems with respect to static perturbations. If the output space of each system is finite-dimensional, then we were able to show a sharp lower bound. Moreover, for regular linear systems with zero feedthrough operator and finite-dimensional output spaces a formula for the stability radius has been developed. An interesting problem for future research is the characterization of the real stability radius and the stability radius with respect to dynamic perturbations.

\bibliographystyle{abbrv}
\bibliography{lit}
\end{document}